\documentclass[a4paper,11pt,fleqn]{amsart}

\usepackage{amssymb,amsmath,amsthm}
\usepackage{centernot}
\usepackage{graphicx}
\usepackage{color}
\usepackage{thmtools}
\usepackage{thm-restate}
\usepackage[shortlabels]{enumitem}
\usepackage{mathrsfs}
\usepackage{mathtools}
\usepackage{stmaryrd}
\usepackage[left=0.9in,right=0.9in,top=1in,bottom=0.9in]{geometry}

\usepackage{graphicx}
\usepackage{xypic}
\usepackage{capt-of}

\usepackage[utf8]{inputenc}
\usepackage[T1]{fontenc}

\newtheorem{theorem}{Theorem}[section]
\newtheorem{lemma}[theorem]{Lemma}
\newtheorem{proposition}[theorem]{Proposition}

\newtheorem{corollary}[theorem]{Corollary}

\theoremstyle{definition}
\newtheorem{definition}[theorem]{Definition}
\newtheorem{example}[theorem]{Example}

\theoremstyle{remark}
\newtheorem{remark}[theorem]{Remark}

\numberwithin{equation}{section}

\theoremstyle{plain}
\newtheorem{theoreml}{Theorem}[section]

\newcommand{\eps}{\varepsilon}

\newcommand{\N}{\mathbb{N}}
\newcommand{\R}{\mathbb{R}}

\newcommand{\EE}{\mathbb{E}}
\newcommand{\OO}{\mathbb{O}}

\newcommand{\aA}{\mathcal{A}}

\newcommand{\xX}{\mathcal{X}}

\DeclareMathOperator{\id}{id}

\newcommand{\ol}{\overline}

\newcommand{\rstr}{\restriction}

\newcommand{\sm}{\setminus}
\newcommand{\sub}{\subseteq}
\DeclareMathOperator{\supp}{supp}
\DeclareMathOperator{\spn}{span}

\newcommand{\ctblsub}[1]{\left[#1\right]^\omega}

\newcommand{\io}{\in\omega}

\newcommand{\wo}{{\wp(\omega)}}
\newcommand{\bo}{{\beta\omega}}

\newcommand{\cso}{\ctblsub{\omega}}

\newcommand{\elli}{{\ell_\infty}}

\title[Symmetric compactifications and separable quotients of spaces $C_p(X)$]{Symmetric compactifications of the integers and separable quotients of spaces $C_p(X)$}
\author[Z. Silber]{Zden\v{e}k Silber}
\address{Institute of Mathematics of the Czech Academy of Sciences, \v{Z}itn\'{a} 25, 115 67 Prague 1, Czech Republic}
\email{zdesil@seznam.cz}

\author[D. Sobota]{Damian Sobota}
\address{Kurt G\"{o}del Research Center, Department of Mathematics, University of Vienna, Kolingasse 14-16, 1090 Vienna, Austria.}
\email{ein.damian.sobota@gmail.com}
\urladdr{www.logic.univie.ac.at/~{}dsobota}

\thanks{The first author was supported by the project L100192501 funded by the ``Programme to support prospective human resources --- post Ph.D. candidates'' of the Czech Academy of Sciences and RVO 67985840. The second author was supported by the Austrian Science Fund (FWF), grants ESP 108-N and 10.55776/PAT4720925. The research presented in this paper was initiated during the first author's visit at the University of Vienna in October 2024, supported by the FWF grant ESP 108-N}

\begin{document}

\begin{abstract}
    A compactification of the discrete space $\omega$ of all integers is called \textit{symmetric} if it is the quotient space obtained by gluing together the remainders of two copies of some other compactification of $\omega$. This is a generalization of both a convergent sequence, which is in a way the minimal symmetric compactification of $\omega$, and the Arkhangel'ski\u{\i}--Bereznitski\u{\i}--Schachermayer space studied in $C_p$-theory, which is in a sense the maximal symmetric compactification of $\omega$. We investigate symmetric compactifications of $\omega$ and their relations to the Separable Quotient Problem for spaces $C_p(X)$ and to the existence of Josefson--Nissenzweig sequences of finitely supported Borel measures on spaces $X$, in particular with supports of bounded size. Further, we reduce the Separable Quotient Problem for spaces $C_p(K)$, $K$ compact, to the case when $K$ is a \textit{totally asymmetric} compactification of $\omega$. Our results shed some new light on the Grothendieck property of Banach spaces $C(K)$.
\end{abstract}


\subjclass[2020]{Primary: 46E10, 54C35, 54D35. Secondary: 46E15, 46E27, 54A20, 54B15.}

\keywords{separable quotients, spaces of continuous functions, pointwise topology, convergence of measures, Josefson--Nissenzweig theorem, compactifications, \v{C}ech--Stone compactification}

\maketitle

\section{Introduction}\label{sec:Intro}

The Separable Quotient Problem, asking whether every infinite-dimensional Banach space admits an infinite-dimensional separable quotient, belongs to the most important long-standing open questions in Banach space theory (see \cite{FKLPS,Muj97}). By the results of Rosenthal \cite{Rosenthal} and Lacey \cite{Lacey}, the problem was settled in affirmative in the class of Banach spaces $C(K)$ of continuous real-valued functions on compact spaces $K$ endowed with the supremum norm (see Section \ref{sec:Prelim} for all unexplained notation). However, if one considers some other classes of locally convex spaces of continuous functions and infinite-dimensional (not necessarily metrizable) separable quotients, then the problem may appear again open---this happens, e.g., in the case of the well-studied class of spaces $C_p(X)$ of continuous real-valued functions on infinite Tychonoff spaces $X$ endowed with the pointwise topology, see e.g. \cite{BKS18,BKS19,KKS26,KMSZ22,KSZ23,KS18,KS26} for details.

It turns out hard not only to decide whether a given space $C_p(X)$ admits a separable quotient, but also, in the case when it is known that it does, to describe any of such quotients. K\k{a}kol and Saxon \cite{KS17} proved that a space $X$ is not pseudocompact (that is, $X$ admits an unbounded continuous real-valued function) if and only if $C_p(X)$ has a quotient isomorphic to the space $\R^\omega$, reducing the problem to the class of pseudocompact spaces $X$. On the other hand, as showed by Banach, K\k{a}kol, and \'{S}liwa \cite{BKS18}, the space $C_p(\bo)$ admits a quotient isomorphic to the space $(\elli)_p$, the standard Banach space $\elli$ endowed with the topology inherited from $\R^\omega$. However, it seems that the most important result so far describing separable quotients of spaces $C_p(X)$ is related to the existence of the following type of sequences of signed real-valued Borel measures on spaces $X$.

\begin{definition}\label{def:JNseq}
    Let $X$ be a Tychonoff space. A normalized sequence $(\mu_n)_{n\io}$ of measures in the dual space $C_p(X)^*$ is a \textit{Josefson--Nissenzweig} sequence, in short a \textit{JN-sequence}, if it is weak* null, i.e. $\lim_{n\to\infty}\mu_n(f)=0$ for every $f\in C_p(X)$.
\end{definition}

The notion of JN-sequences is of course closely related to the classical Josefson--Nissenzweig theorem (\cite{Josef,Nissen}). Banach, K\k{a}kol, and \'{S}liwa \cite{BKS19} proved that a space $C_p(X)$ admits a quotient isomorphic to the space $(c_0)_p$, i.e. the Banach space $c_0$ endowed with the pointwise topology inherited from $\R^\omega$, if and only if $X$ carries a JN-sequence. Consequently, if $X$ contains a non-trivial convergent sequence (e.g. $X$ is compact metrizable), then $C_p(X)$ has a quotient isomorphic to $(c_0)_p$. On the other hand, it was also observed in \cite{BKS19} that the \v{C}ech--Stone compactification $\bo$ of the space $\omega$ of all natural numbers does not carry any JN-sequence and so $C_p(\bo)$ does not admit any quotients isomorphic to $(c_0)_p$. Further results and examples concerning the existence of JN-sequences on compact spaces $K$ were presented in \cite{KMSZ22,KSZ23,KSZ24,MS24,MSZ24}.


To investigate further the Separable Quotient Problem as well as possible forms of separable quotients of spaces $C_p(X)$, Banakh, K\k{a}kol, Kurka, and \'{S}liwa \cite{BKS18,KKS26,KS18} introduced and study a property of spaces $X$ displaying some form of self-similarity, called the Two-Disjoint-Copies Property (see Definition \ref{def:2dcp}). In their works, they exploited among others the following notion.

\begin{definition}\label{def:sqseq}
    Let $X$ be a Tychonoff space. A normalized sequence $(\mu_n)_{n\io}$ of measures in $C_p(X)^*$ is an \textit{(sq)-sequence} if
    \[C_p(X) = \overline{\bigcup_{m \in \omega} \bigcap_{n \geq m} \ker(\mu_n)}^{C_p(X)}.\]
\end{definition}

As observed in \cite{KS93,KS18}, a space $C_p(X)$ admits a separable quotient if and only if $X$ carries an (sq)-sequence. Consequently, (sq)-sequences constitute a central, though not always handy, concept for the study of separable quotients.  

One of the goals of this paper is to investigate when the presence of an (sq)-sequence $(\mu_n)_{n\io}$ on a compact space $K$ yields that $K$ actually carries a JN-sequence and so that $C_p(K)$ has a quotient isomorphic to $(c_0)_p$. Here, our main working assumption is that the given (sq)-sequence $(\mu_n)_{n\io}$ has uniformly bounded sizes of supports, that is, there is $M\io$ such that $|\supp(\mu_n)|\le M$ for all $n\io$. It appears that in this case $K$ also carries a JN-sequence with uniformly bounded sizes of supports and so, by the work \cite{MSZ24}, even a JN-sequence with supports of size $2$. Moreover, $K$ must then also contain a compactification $b\omega$ of $\omega$ which is symmetric in the following sense.

\begin{definition}\label{def:symmetric}
    A compactification $b\omega$ of $\omega$ is \textit{symmetric} if there exist a partition $(A,B)$ of $\omega$ into two infinite subsets $A,B$ and a homeomorphism $h\colon\overline{A}^{b\omega}\to\overline{B}^{b\omega}$ satisfying $h \restriction\big(\overline{A}^{b\omega}\setminus A\big) = \id_{\overline{A}^{b\omega}\setminus A}$.
\end{definition}

Intuitively and as both the definition and the name suggest, a symmetric compactification $b\omega$ is such a compactification which presents a lot of symmetry with respect to their remainder $b\omega\sm\omega$. Roughly speaking, the compactification $b\omega$ is symmetric if it is the quotient space obtained by gluing together the remainders of two copies of the same compactification $\gamma\omega$ of $\omega$ (cf. Proposition \ref{prop:symmetric_two_copies}). For various other characterizations of symmetric compactifications of $\omega$, see Theorem \ref{theorem:SymmetricCase}.

Symmetric compactifications of $\omega$ are quite ubiquitous. For example, the ordinal interval $[0,\omega]$, which could be identified with a convergent sequence together with its limit point, is the simplest instance of a symmetric compactification of $\omega$ (see Example \ref{ex:conv}). On the other hand, the space $K_{ABS}$, introduced and studied by Arkhangel’ski\u{\i}, Bereznitski\u{\i} \cite{Ber71}, and Schachermayer \cite{Sch82} in the context of $C_p(X)$-spaces and $C(K)$-spaces, can be thought of as the most complex symmetric compactification of $\omega$ (see Example \ref{ex:ABS}). Moreover, any separable compact space can be the remainder of a symmetric compactification (see Proposition \ref{prop:arbremain}).

Let us finally state the first main result of this paper (for its more general statement and the proof, see Theorem \ref{theorem:sq_jn_2_k}):

\begin{theoreml}\label{theorem:mainA}
    For a compact space $K$, the following statements are equivalent:
    \begin{enumerate}[(1)]
        \item $K$ carries an (sq)-sequence $(\mu_n)_{n\io}$ such that $|\supp(\mu_n)|\le M$ for all $n\io$ and some $M\ge2$.
        \item $K$ carries a JN-sequence $(\mu_n)_{n\io}$ such that $|\supp(\mu_n)|=2$ for all $n\io$.
        \item $K$ contains a copy of a symmetric compactification $b\omega$ of $\omega$.
    \end{enumerate}
\end{theoreml}

Combining Theorem \ref{theorem:mainA} with the aforementioned theorem of Banakh, K\k{a}kol, and \'{S}liwa \cite{BKS19}, we immediately get the following result describing quotients of spaces $C_p(K)$ carrying (sq)-sequences with uniformly bounded sizes of supports.

\begin{theoreml}\label{theorem:mainB}
    If a compact space $K$ carries an (sq)-sequence $(\mu_n)_{n\io}$ such that $|\supp(\mu_n)|\le M$ for all $n\io$ and some $M\io$, then $C_p(K)$ admits a quotient isomorphic to $(c_0)_p$.
\end{theoreml}

Note that not every compact space carrying a JN-sequence carries one with uniformly bounded sizes of supports---see \cite{MS24,MSZ24} for suitable examples. It follows that symmetric compactifications of $\omega$ do not fully characterize admitting $(c_0)_p$ as a quotient of a space $C_p(K)$.

The existence of a quotient of a given space $C_p(K)$ which is isomorphic to $(c_0)_p$ implies that the Banach space $C(K)$ is not \textit{Grothendieck}, that is, there is a normalized weak* null sequence in its dual space $C(K)^*$ which is not weakly null, or, equivalently, if there is a separable Banach space $Y$ and a non-weakly compact operator $T\colon C(K)\to Y$ (see \cite{KKLPS,KSZ23}). It is a long-standing open problem, posed by Diestel \cite{Die73}, to characterize those compact spaces $K$ for which their spaces $C(K)$ are Grothendieck only in terms of the topology of $K$. For examples of various sufficient criteria, related to completeness and separation properties of the Boolean algebras of clopen subsets of $K$, see \cite{KKLPS}. On the other hand, the paper \cite{MS24} presents numerous examples of \textit{forbidden structures} or, more precisely, such countable topological spaces $Y$ with one non-isolated point which have the following property: if $K$ contains a homeomorphic copy of $Y$, then $C(K)$ is not Grothendieck. A typical example of such a forbidden structure, and the only compact one, is the space $[0,\omega]$, which, as stated above, is also an example of the simplest symmetric compactification. Consequently, the present work supplements \cite{MS24} in providing numerous examples of \textit{compact} forbidden structures, implying the lack of the Grothendieck property of the respective spaces $C(K)$.

\begin{theoreml}\label{theorem:mainC}
    If a compact space $K$ contains a copy of a symmetric compactification of $\omega$, then the space $C(K)$ is not Grothendieck.
\end{theoreml}

Our last main result concerns compactifications which are as far from being symmetric as possible.

\begin{definition}\label{def:totasym}
    A compactification $b \omega$ of $\omega$ is \textit{totally asymmetric} if for every two disjoint infinite subsets $A,B \in\omega$ their closures $\overline{A}^{b\omega}$ and $\overline{B}^{b\omega}$ are not homeomorphic.
\end{definition}

For constructions of (consistent) examples of totally asymmetric compactifications, see \cite{Bac18,KS18}. Note that the notion is closely related to the aforementioned Two-Disjoint-Copies Property (see Proposition \ref{prop:KakSli2018}).

We show that the general Separable Quotient Problem for $C_p(K)$-spaces, $K$ compact, can be reduced to investigating whether spaces $C_p(b\omega)$ admit separable quotients for totally asymmetric compactifications $b\omega$ of $\omega$, improving \cite[Proposition 19]{KS18}. 

\begin{theoreml}\label{theorem:mainD}
    The following statements are equivalent:
    \begin{enumerate}[(i)]
        \item There is a compact space $K$ such that $C_p(K)$ does not have a separable quotient.
        \item There is a totally asymmetric compactification $b \omega$ of $\omega$ such that $C_p(b \omega)$ does not have a separable quotient.
    \end{enumerate}
\end{theoreml}

For the proof of the theorem, see Theorem \ref{theorem:totasym}. Let us note that from the results mentioned above it follows that if, for a compact space $K$, the space $C_p(K)$ does not admit a separable quotient, then $K$ must be an \textit{Efimov} space, i.e. an infinite compact space which contains neither a convergent sequence nor a copy of $\beta \omega$; such spaces $K$ have been so far constructed only consistently and it is not known whether any of those examples $K$ is such that $C_p(K)$ does not admit a separable quotient, see e.g. \cite{KSZ23,KS18} for details.

\medskip

The paper is organized as follows. In Section \ref{sec:Prelim} we recall the notation and necessary definitions for the rest of the paper. In Section \ref{sec:Sequences} we investigate weak* closures of countable subsets of $C_p(X)^*$ and introduce so-called simple JN-sequences. Section \ref{sec:Sym} is devoted to study properties of symmetric compactifications of $\omega$ and their relations to JN-sequences and (sq)-sequences. Finally, in Section \ref{sec:Asym} we study totally asymmetric compactifications of $\omega$ and their relations to the lack of separable quotients in spaces $C_p(X)$.

\section{Preliminaries} \label{sec:Prelim}

For a set $X$, by $|X|$ we denote its cardinality and by $\wp(X)$ and $\ctblsub{X}$ the collections of all subsets of $X$ and of all countable infinite subsets of $X$, respectively. The identity function on $X$ is denoted by $\id_X$.

By $\omega$ we denote the first infinite limit ordinal number. As usual, we identify $\omega$ with the discrete space of all natural numbers. Note that for any compactification $b\omega$ of $\omega$ and $n\in\omega$ the singleton $\{n\}$ is a clopen subset of $b\omega$, i.e. the point $n$ is isolated in $b\omega$. We also set $\EE=\{2n\colon n\io\}$ and $\OO=\{2n+1\colon n\io\}$. A pair $(A,B)$ of subsets of $\omega$ is a \textit{partition} of $\omega$ if $A,B\in\cso$, $\omega=A\cup B$, and $A\cap B=\emptyset$.

We work exclusively with topological spaces which are Tychonoff (so, in particular, compact spaces are assumed to be normal). Let $X$ be a set. If $\tau$ is a topology on $X$, then for a set $A\sub X$ we denote its closure in $X$ with respect to $\tau$ by $\ol{A}^{(X,\tau)}$, $\ol{A}^\tau$, or $\ol{A}^X$, or even briefly by $\ol{A}$ if the topology is clear from the context. If $\sim$ is a closed equivalence relation on a topological space $X$, then $X/_\sim$ denotes the quotient space of $X$ modulo $\sim$. For two topological spaces $X$ and $Y$ we define their disjoint union by $X\sqcup Y=(X\times\{0\})\cup(Y\times\{1\})$, equipped with the standard topology. 

Let $X$ be a topological space. By $C(X)$ we denote the set of all continuous real-valued functions on $X$. By $C_p(X)$ we denote the set $C(X)$ endowed with the pointwise topology $\tau_p$ inherited from the space $\R^X$. For any $f\in C_p(X)$, $x_1,\ldots,x_n\in X$, and $\eps>0$, we set
\[V(f;x_1,\ldots,x_n;\eps)=\{g\in C_p(X)\colon\ |f(x_i)-g(x_i)|<\eps\text{ for every }i=1,\ldots, n\big\},\]
so $V(f;x_1,\ldots,x_n;\eps)$ is a standard subbasic subset of $C_p(X)$. If $X$ is compact, then $C(X)$ is by default endowed with the supremum norm $\|\cdot\|_\infty$, which makes it a Banach space. 

By $Bor(X)$ we denote the $\sigma$-field of all Borel subsets of $X$ and by $Clopen(X)$ the Boolean algebra of all clopen subsets of $X$. A function $\mu\colon Bor(X)\to\R$ is a \textit{measure} on $X$ if $\mu$ is $\sigma$-additive, inner regular with respect to compact sets, and outer regular with respect to open sets. Note that such functions are usually called \textit{Radon measures}, we drop the word Radon for the sake of brevity. For a point $x\in X$ by $\delta_x$ we denote the one-point measure concentrated at $x$ (i.e. \textit{Dirac delta} at $x$). The \textit{total variation} $\|\mu\|$ of $\mu$ is defined as follows:
\[\|\mu\|=\sup\big\{|\mu(A)|+|\mu(B)|\colon\ A,B\in Bor(X), A\cap B=\emptyset\big\}.\]
The set of all measures on $X$ is denoted by $M(X)$. Recall that $(M(X),\|\cdot\|)$ is a Banach space. We also set
\[M^1(X)=\big\{\mu\in M(X)\colon \|\mu\|\le1\big\}.\]
A subset $A\sub M(X)$ is \textit{normalized} if $\|\mu\|=1$ for every $\mu\in A$.

A measure $\mu$ on $X$ is said to be \textit{finitely supported} if there are $n\io$, a sequence $x_1,\ldots,x_n$ of distinct points in $X$, and a sequence $\alpha_1,\ldots,\alpha_n$ of non-zero real numbers such that $\mu=\sum_{i=1}^n\alpha_i\delta_{x_i}$; in this case we have $\|\mu\|=\sum_{i=1}^n|\alpha_i|$. The set of all finitely supported measures on $X$ is denoted by $M_f(X)$; put also $M_f^1(X)=M_f(X)\cap M^1(X)$. For each $\mu\in M_f(X)$ we define the following three sets:
\begin{itemize}
    \item $\supp(\mu)=\{x\in X\colon \mu(\{x\})\neq 0\}$ (the \textit{support} of $\mu$),
    \item $\supp^+(\mu)=\{x\in X\colon \mu(\{x\})>0\}$,
    \item $\supp^-(\mu)=\{x\in X\colon \mu(\{x\})<0\}$.
\end{itemize}
Of course, $\supp(\mu)=\supp^+(\mu)\cup\supp^-(\mu)$. For each $k\io$ we also set 
\[M_k(X)=\big\{\mu\in M_f(X)\colon |\supp(\mu)|\le k\big\}\]
and $M_k^1(X)=M_k(X)\cap M^1(X)$. 

We say that a sequence $(\mu_n)_{n\io}$ in $M_f(X)$ is \textit{disjointly supported} if $\supp(\mu_n)\cap\supp(\mu_k)=\emptyset$ for every $n\neq k\io$, and that $(\mu_n)_{n\io}$ is \textit{discretely supported} if $\bigcup_{n\io}\supp(\mu_n)$ is a discrete subset of $X$.

If $X$ is compact, then by the Riesz--Markov--Kakutani theorem the space $M(X)$ is identified with the dual Banach space $C(X)^*$, with the action of $M(X)$ on $C(X)$ given for every $\mu\in M(X)$ by the formula
\[C(X)\ni f\longmapsto\int_X fd\mu.\]
Similarly, the topological dual $C_p(X)^*$ of the (locally convex) space $C_p(X)$ is identified with the linear space $M_f(X)$, with the action of $M_f(X)$ on $C_p(X)$ given for every $\mu\in M_f(X)$, $\mu=\sum_{i=1}^n\alpha_i\delta_{x_i}$, by the formula
\[C_p(X)\ni f\longmapsto\int_X fd\mu=\sum_{i=1}\alpha_i f(x_i).\]
By the weak* topology $w^*$ on $M_f(X)$ (resp. $M(X)$ if $X$ is compact) we mean the topology induced by $C_p(X)$ (resp. $C(X)$) via the above action. Basic weak* neighbourhoods of any measure $\mu\in M_f(X)$ are then given by the formula
\[V(\mu;f_1,\ldots,f_n;\eps)=\big\{\nu\in M_f(X)\colon\ |\mu(f_i)-\nu(f_i)|<\eps\text{ for every }i=1,\ldots, n\big\},\]
where $f_1,\ldots,f_n\in C_p(X)$ and $\eps>0$, so $V(\mu;f_1,\ldots,f_n;\eps)$ is a standard subbasic subset of $(M_f(X),w^*)$.

The (real) Banach spaces $c_0$, $\ell_1(\Gamma)$, and $\elli(\Gamma)$ are defined in the standard way. By $(c_0)_p$ we mean the space $c_0$ endowed with the pointwise topology $\tau_p$ inherited from $\R^\omega$. The weak topology of any Banach space $E$ will be denoted by $w$.

\section{Sequences in $C_p(K)^*$ and the zero measure $0$} \label{sec:Sequences} 

In this warm-up section we provide some basic criteria for the zero measure $0$ to belong to the weak* closures of countable sets of measures with supports of size $2$ as well we prove that so-called simple JN-sequences are (sq)-sequences. 
%
We start with the following proposition.

\begin{proposition}\label{prop:zero}
    Let $X$ be a Tychonoff space containing a countable infinite discrete set $D$ whose closure $\ol{D}$ in $X$ is compact. Then, there is a sequence $(\mu_n)_{n\io}$ in $M_f(X)$ such that
    \begin{enumerate}[(a)]
        \item $\|\mu_n\|=1$, $\supp(\mu_n)\sub D$, and $|\supp(\mu_n)|=2$ for every $n\io$,
        \item $0\in\ol{\{\mu_n\colon n\io\}}^{w^*}$.
    \end{enumerate}
\end{proposition}

\begin{proof}
    Let $F=\ol{D}^X$. We first derive the conclusion for $C_p(F)^*$ instead of $C_p(X)^*$. We will consider two cases.

    Assume first that $F=\beta D$. For each distinct $a,b\in D$, let
    \[\nu_{a,b}=\frac{1}{2}(\delta_a-\delta_b).\]
    Of course, every $\nu_{a,b}$ trivially satisfies the conditions listed in (a).

    Each measure $\nu_{a,b}$ can be treated in a natural way as an element $x_{a,b}$ of the Banach space $\ell_1(D)$. Then, the weak* topology on the set $\{\nu_{a,b}\colon a\neq b\in D\}$ in $C_p(F)^*$ coincides with the weak topology on $\{x_{a,b}\colon a\neq b\in D\}$ in $\ell_1(D)$ (this is because the weak topology on $\ell_1(D)$ is generated by $\ell_\infty(D)$, but $\ell_\infty(D)$ can be canonically identified with $C_p(F)$ as any bounded function on $D$ admits a unique extension to $\beta D = F$). From \cite[Lemma 3.2]{GKS17} it follows that 
    \[0\in\ol{\{x_{a,b}\colon\ a\neq b\in D\}}^w,\]
    and so
    \[0\in\ol{\{\nu_{a,b}\colon\ a\neq b\in D\}}^{w^*}.\]
    Since the set $D$ is countable, we can simply enumerate $\{\nu_{a,b}\colon a\neq b\in D\}$ as $({\mu_n})_{n\io}$.

\medskip

    Suppose now that $F\neq\beta D$. There is a continuous surjection $\varphi\colon\beta D\to F$ extending the identity $\id_D$. Let $({\mu_n})_{n\io}$ be the sequence of measures on $\beta D$ defined as in the first part of the proof. Let $T\colon C_p(\beta D)^*\to C_p(F)^*$ be the operator adjoint to the composition operator with $\varphi$, that is, it is induced by the assignment $T\delta_x=\delta_{\varphi(x)}$ for every $x\in\beta D$. For each $n\io$ set $\theta_n=T(\mu_n)$. Since $\varphi\rstr D=\id_D$, condition (a) is still satisfied for the sequence $({\theta_n})_{n\io}$. Further, (b) holds as $T$, being an adjoint operator, is weak*--weak* continuous.

\medskip

    To get the conclusion for $C_p(X)^*$, let $({\mu_n})_{n\io}$ be a sequence in $M_f(F)$ satisfying conditions (a) and (b). Let $S\colon C_p(F)^*\to C_p(X)^*$ be the operator adjoint to the restriction operator to $F$, that is, it is given by $S(\mu)(f)=\mu(f\rstr F)$ for every $\mu\in C_p(F)^*$ and $f\in C_p(X)$. For each $n\io$ set $\rho_n=S(\mu_n)$. Then, $({\rho_n})_{n\io}$ satisfies (a) trivially. Moreover, again as $S$ is weak*--weak* continuous (being the adjoint operator), it satisfies (b), too.  
\end{proof}

\begin{lemma}\label{lem:bo_weakstar}
    For any infinite co-infinite subset $A$ of $\omega$ we have
    \[0\not\in\ol{\{\delta_{n}-\delta_{m}\colon\ n\in A, m\in\omega\sm A\}}^{(C_p(\bo)^*,w^*)}.\]
\end{lemma}
\begin{proof}
    Set $U=\ol{A}^{\bo}$. Then, $U$ is a clopen subset of $\bo$ and so the characteristic function $\chi_U$ of $U$ in $\bo$ is continuous, that is, $\chi_U\in C_p(\bo)$. The set
    \[V=\{\mu\in C_p(\bo)^*\colon |\mu(\chi_U)|<1\}\]
    is a weak* open neighbourhood of $0$. However, since $(\delta_n-\delta_m)(\chi_U)=1$ for every $n\in A$ and $m\in\omega\sm A$, we get
    \[V\cap\{\delta_{n}-\delta_{m}\colon\ n\in A, m\in\omega\sm A\}=\emptyset,\]
    which finishes the proof.
\end{proof}

\begin{proposition}\label{prop:zero2}
Let $D$ be a countable infinite discrete subset of a compact space $K$. Then, there is an infinite subset $A\sub D$ such that $D\sm A$ is infinite too and 
\[0\in\ol{\{\delta_{a}-\delta_{b}\colon\ a\in A, b\in D\sm A\}}^{(C_p(K)^*,w^*)}\]
if and only if $\ol{D}^K\neq\beta D$.
\end{proposition}
\begin{proof}
    Assume that $\ol{D}^K=\beta D$. Then, by Lemma \ref{lem:bo_weakstar}, for every infinite subset $A$ of $D$ such that $D\sm A$ is infinite we have
    \[0\not\in\ol{\{\delta_{a}-\delta_{b}\colon\ a\in A, b\in D\sm A\}}^{(C_p(\ol{D}^K)^*,w^*)}\]
    and hence, by the Tietze extension theorem,
    \[0\not\in\ol{\{\delta_{a}-\delta_{b}\colon\ a\in A, b\in D\sm A\}}^{(C_p(K)^*,w^*)}.\]
    Consequently, implication $\Rightarrow$ holds.
    
    To prove implication $\Leftarrow$, assume that $\ol{D}^K\neq\beta D$. There are an infinite subset $A$ of $D$ and a point $c\in\ol{A}^K\cap\ol{D\sm A}^K$. Note that $D\sm A$ is infinite, too. We claim that
    \[0\in\ol{\{\delta_a-\delta_b\colon\ a\in A,b\in D\sm A\}}^{(C_p(K)^*,w^*)}.\]
    To see this, let $U$ be a weak* open neighbourhood of $0$ in $C_p(K)^*$. There is a convex open subset $V$ of $U$ such that $0\in V$, $V=-V$, and $V+V\sub U$. Since $\delta_c+V$ is an open neighbourhood of $\delta_c$, there are $a\in A$ and $b\in D\sm A$ such that $\delta_a,\delta_b\in\delta_c+V$. Consequently, $\delta_a-\delta_c\in V$ and $\delta_c-\delta_b\in-V=V$, and hence
    \[(\delta_a-\delta_c)+(\delta_c-\delta_b)\in V+V\sub U,\]
    that is, $\delta_a-\delta_b\in U$.
\end{proof}

\subsection{Simple JN-sequences and (sq)-sequences}

It follows from the above propositions that it is not hard to find a countable set of normalized elements of a given space $M_2(X)$ that has $0$ in its weak* closure. However, nothing guarantees that this countable set is actually a weak* convergent sequence and this also motivated introducing the notion of JN-sequences (Definition \ref{def:JNseq}). Surprisingly, it turned out that they have a particularly good feature of being resilient to various modifications, and so, as shown in \cite{MSZ24}, if a space $M_f(X)$ admits JN-sequences, then they can always be chosen to be of a particularly nice type.

\begin{definition}\label{def:JNseq_simple}
    Let $X$ be a Tychonoff space. We say that a JN-sequence $(\mu_n)_{n\io}$ of measures on $X$ is \textit{simple} if the following conditions hold:
    \begin{itemize}
        \item $(\mu_n)_{n\io}$ is disjointly supported,
        \item $(\mu_n)_{n\io}$ is discretely supported,
        \item $\mu_n\big(\supp^+(\mu_n)\big)=1/2$ and $\mu_n\big(\supp^-(\mu_n)\big)=-1/2$ for every $n\io$.
    \end{itemize}
\end{definition}

\begin{theorem}[{\cite[Theorem 1.2]{MSZ24}}]\label{theorem:JNseq_simple}
    If a Tychonoff space $X$ carries a JN-sequence, then $X$ carries a simple JN-sequence.
\end{theorem}

%

%
By the following proposition simple JN-sequences are (sq)-sequences.

\begin{proposition}\label{prop:simple_jn_sq}
    Let $X$ be a Tychonoff space and let $(\mu_n)_{n\io}$ be a JN-sequence on $X$. If $(\mu_n)_{n\io}$ is simple, then $(\mu_n)_{n\io}$ is an (sq)-sequence.
\end{proposition}
\begin{proof}
    Assume that $(\mu_n)_{n\io}$ is simple. Let $T\colon C_p(X)\to(c_0)_p$ be defined for every $f\in C_p(X)$ as follows:
    \[S(f)=\big(\mu_n(f)\big)_{n\io}.\]
    It follows by the simplicity of $(\mu_n)_{n\io}$ and the proof of \cite[Theorem 1]{BKS19} that $S$ is an open continuous linear surjection onto $(c_0)_p$.

    For each $n\io$ let $e_n$ denote the $n$-th standard unit vector in $c_0$. The set 
    \[\spn\{e_n\colon\ n\io\}=\bigcup_{m\io}\spn\{e_n\colon\ n<m\}\]
    is dense in $(c_0)_p$, therefore the set
    \[S^{-1}\big[\spn\{e_n\colon\ n\io\}\big]=\bigcup_{m\io}S^{-1}\big[\spn\{e_n\colon\ n<m\}\big]\]
    is dense in $C_p(X)$. For every $m\io$ and $f\in S^{-1}\big[\spn\{e_n\colon\ n<m\}\big]$, we have $S(f)\in\spn\{e_n\colon\ n<m\}$, so $\mu_n(f)=S(f)(n)=0$ for every $n\ge m$, that is, $f\in\ker(\mu_n)$ for every $n\ge m$. Consequently, for every $m\io$ it holds
    \[S^{-1}\big[\spn\{e_n\colon\ n<m\}\big]\sub\bigcap_{n\ge m}\ker(\mu_n),\]
    hence
    \[\bigcup_{m\io}S^{-1}\big[\spn\{e_n\colon\ n<m\}\big]\sub\bigcup_{m\io}\bigcap_{n\ge m}\ker(\mu_n),\]
    and thus the set $\bigcup_{m\io}\bigcap_{n\ge m}\ker(\mu_n)$ is dense in $C_p(X)$. It follows that $(\mu_n)_{n\io}$ is an (sq)-sequence.
\end{proof}

Combining Proposition \ref{prop:simple_jn_sq} with Theorem \ref{theorem:JNseq_simple}, we get the following corollary.

\begin{corollary}\label{cor:jn_implies_simple_jn_sq}
    If a Tychonoff space $X$ carries a JN-sequence, then $X$ carries a sequence of measures which is simultaneously a simple JN-sequence and an (sq)-sequence.
\end{corollary}

\section{Symmetric compactifications of the integers} \label{sec:Sym}

In this section we investigate properties of the central notion of the paper, symmetric compactifications of $\omega$. In particular, in Theorem \ref{theorem:SymmetricCase} we provide their various characterizations in terms of JN-sequences and (sq)-sequences, and then we use those characterizations to prove the main result of this paper, Theorem \ref{theorem:sq_jn_2_k} (essentially extending Theorem \ref{theorem:mainA} from Introduction).
We start with two basic examples.

\begin{example}\label{ex:conv}
    Let the ordinal interval $b\omega=\omega\cup\{\omega\}=[0,\omega]$ be endowed with the order topology. Set $A=\EE$ and $B=\OO$. Then, $\ol{A}^{b\omega}=A\cup\{\omega\}$ and $\ol{B}^{b\omega}=B\cup\{\omega\}$. For each $n\io$ let $h(2n)=2n+1$, and $h(\omega)=\omega$. It follows that $h\colon\ol{A}\to\ol{B}$ is a homeomorphism such that $h\rstr(\ol{A}\sm A)=\id_{\ol{A}\sm A}$. Consequently, $[0,\omega]$ is a symmetric compactification of $\omega$.

    The above argument can be easily adapted to work for any countable ordinal $\alpha\ge\omega$ and the interval $[0,\alpha]$.
\end{example}

Directly from Definition \ref{def:symmetric} it follows that if a partition $(A,B)$ of $\omega$ witnesses that a compactification $b\omega$ is symmetric, then $\overline{A}^{b\omega}\sm A=\overline{B}^{b\omega}\sm B=b\omega\sm\omega$. Consequently, if a compactification $\gamma\omega$ of $\omega$ is such that $\big(\overline{A}^{\gamma\omega}\sm A\big)\triangle\big(\overline{B}^{\gamma\omega}\sm B\big)\neq\emptyset$ for any partition $(A,B)$ of $\omega$, then $\gamma\omega$ is not symmetric.

\begin{example}\label{ex:bo}
    $\bo$ is not a symmetric compactification of $\omega$, since, by the standard properties of $\bo$, for any partition $(A,B)$ of $\omega$ we have $\overline{A}^{\bo}\cap\overline{B}^{\bo}=\emptyset$.
\end{example}

In order to consider some further examples of symmetric compactifications, we will need the following piece of notation.

\begin{definition}
    For a compactification $\gamma\omega$ of $\omega$ we define the equivalence relation $\sim_{\gamma\omega}$ on $\gamma\omega\sqcup\gamma\omega$ as the relation having precisely the following classes: $\{(n,i)\}$ for all $n\io$ and $i\in\{0,1\}$, and $\big\{(x,0),(x,1)\big\}$ for all $x\in\gamma\omega\sm\omega$.
\end{definition}

For a compactification $\gamma\omega$ of $\omega$, the quotient space $(\gamma\omega\sqcup\gamma\omega)/_{\sim_{\gamma\omega}}$ may be naturally considered as the space obtained by gluing together in the natural way two copies of the remainder $\gamma\omega\sm\omega$. Also, $(\gamma\omega\sqcup\gamma\omega)/_{\sim_{\gamma\omega}}$ is itself a compactification of $\omega$. 

As one may easily suspect, each symmetric compactification of $\omega$ is in fact of the form $(\gamma\omega\sqcup\gamma\omega)/_{\sim_{\gamma\omega}}$ for some $\gamma\omega$.

\begin{proposition}\label{prop:symmetric_two_copies}
    Let $b\omega$ be a compactification of $\omega$. Let $(A,B)$ be a partition of $\omega$. The following statements are equivalent:
    \begin{enumerate}[(i)]
        \item $b\omega$ is symmetric with respect to the partition $(A,B)$ and some homeomorphism $h\colon\overline{A}^{b\omega}\to\overline{B}^{b\omega}$ such that $h\rstr\big(\overline{A}^{b\omega}\sm A\big)=\id_{\overline{A}^{b\omega}\sm A}$.
        \item There are a compactification $\gamma\omega$ of $\omega$ and homeomorphisms $h_A\colon\overline{A}^{b\omega}\to\gamma\omega$ and $h_B\colon\overline{B}^{b\omega}\to\gamma\omega$ such that the mapping $H\colon b\omega\to(\gamma\omega\sqcup\gamma\omega)/_{\sim_{\gamma\omega}}$, given by the formula $H(x)=\big[(h_A(x),0)\big]_{\sim_{\gamma\omega}}$ for each $x\in\overline{A}^{b\omega}$ and $H(x)=\big[(h_B(x),1)\big]_{\sim_{\gamma\omega}}$ for each $x\in\overline{B}^{b\omega}$, is a homeomorphism.
    \end{enumerate}
\end{proposition}
\begin{proof}
    The proof is quite routine, so we skip most details. For implication (i)$\Rightarrow$(ii), note that since $A$ is a countably infinite discrete set, we can identify $\overline{A}^{b \omega}$ with a compactification $\gamma \omega$ of $\omega$ by a homeomorphism $h_A\colon\overline{A}^{b\omega}\to\gamma\omega$. Let $h_B\colon\overline{B}^{b\omega}\to\gamma\omega$ be defined as $h_B=h_A \circ h^{-1}$. Checking that the mapping $H$ defined as in (ii) is a homeomorphism is now routine.

    For implication (ii)$\Rightarrow$(i), it is enough to show that the mapping $h_B^{-1}\circ h_A$ is a homeomorphism satisfying $(h_B^{-1}\circ h_A)\rstr\big(\overline{A}^{b\omega}\sm A\big)=\id_{\overline{A}\sm A}$---this however follows immediately from the observation that $H(x)=H\big((h_B^{-1}\circ h_A)(x)\big)$ for every $x\in\overline{A}^{b\omega}\sm A$.
\end{proof}

It appears that symmetric compactifications of $\omega$ are in fact easy to construct and so they are quite abundant.

\begin{proposition}\label{prop:arbremain}
    Let $\gamma \omega$ be a compactification of $\omega$. Then, the remainders $\gamma \omega \setminus \omega$ and $\big((\gamma\omega\sqcup\gamma\omega)/_{\sim_{\gamma\omega}}\big)\setminus\omega$ are homeomorphic.

    In particular, for a compact space $K$ the following conditions are equivalent:
    \begin{enumerate}[(i)]
        \item $K$ is the remainder of a symmetric compactification of $\omega$,
        \item $K$ is the remainder of a compactification of $\omega$,
        \item $K$ is a continuous image of the \v{C}ech--Stone remainder $\beta \omega \setminus \omega$.
    \end{enumerate}
\end{proposition}

\begin{proof}
    The first statement follows from the definition of $(\gamma\omega\sqcup\gamma\omega)/_{\sim_{\gamma\omega}}$. Indeed, let $Q\colon\gamma\omega\sqcup\gamma\omega\to(\gamma\omega\sqcup\gamma\omega)/_{\sim_{\gamma\omega}}$ be the canonical quotient map. Then, the restriction $Q \rstr (\gamma\omega\setminus\omega)\times\{0\}$ is a continuous injection from the compact space $(\gamma\omega\setminus\omega)\times\{0\}$, which is clearly homeomorphic to $\gamma\omega\setminus\omega$, onto the space $\big((\gamma\omega \sqcup\gamma\omega)/_{\sim_{\gamma\omega}}\big)\setminus\omega$, and hence $Q \rstr (\gamma\omega\setminus\omega)\times\{0\}$ is a homeomorphism. 

    Consequently, the class of remainders of symmetric compactifications of $\omega$ coincides with the class of all remainders of compactifications of $\omega$, which is well-known to coincide with continuous images of the \v{C}ech--Stone remainder $\beta\omega\setminus\omega$.
\end{proof}

\begin{remark}
    Note that any non-empty separable compact space $K$ as well as any non-empty compact space $K$ of weight at most $\aleph_1$ is a continuous image of the space $\beta\omega\setminus\omega$ (see \cite[Theorem 1.3.3]{vMillHBK}), and so $K$ can be the remainder of a symmetric compactification of $\omega$. For more information concerning spaces being continuous images of $\bo\sm\omega$, see e.g. the works of van Douwen and Przymusi\'{n}ski \cite{Prz82,vDP80}.
\end{remark}

The following example of a compactification of $\omega$ was in fact the main motivation for us to introduce the notion of symmetric compactifications. It was first studied by Arkhangel’ski\u{\i}, Bereznitski\u{\i} \cite{Ber71}, and Schachermayer \cite{Sch82}, and has numerous important applications in $C_p$-theory and $C(K)$-space theory. Its universal role in our study of symmetric compactifications will be explained in Theorem \ref{theorem:SymmetricCase} below.

\begin{example} \label{ex:ABS}
    \textit{The Arkhangel’ski\u{\i}--Bereznitski\u{\i}--Schachermayer space $K_{ABS}$.} The space is constructed as follows. 
    
    First, for each $n\in\EE$ let $h(n)=n+1$. Then, extend the mapping $h\colon\EE\to\OO$ to the (unique) homeomorphism $h\colon\beta\EE\to\beta\OO$. Note that $\beta\EE\cap\beta\OO=\emptyset$. Finally, on the space $\bo=\beta\EE\sqcup\beta\OO$ define the equivalence relation $\sim$ by declaring its classes to be: $[n]_\sim=\{n\}$ for each $n\io$, and $[x]_\sim=\{x,h(x)\}$ for $x\in\beta\EE\sm\EE$. Set $K_{ABS}=\bo/_\sim$.

    Equivalently, $K_{ABS}$ can be defined as the Stone space $St(\aA_{ABS})$ of the Boolean subalgebra $\aA_{ABS}$ of $\wo$ consisting of all those sets $A\in\wo$ for which it holds:
    \[2k\in A\ \Longleftrightarrow\ 2k+1\in A,\]
    for all but finitely many $k\io$. Note that the Boolean algebras $\aA_{ABS}$ and $Clopen(K_{ABS})$ are isomorphic.

    The space $K_{ABS}$ is clearly a compactification of $\omega$. It has already appeared to be useful in studies of quotients of $C_p(X)$-spaces, the Grothendieck property of Banach spaces $C(K)$, and JN-sequences on compact spaces, see e.g. \cite{Ark87,KS13,Mar97,MSZ24,Sch82}. Its most important features are:
    \begin{enumerate}[(1)]
        \item For any (isolated) point $n\in \omega\sub K_{ABS}$, the spaces $K_{ABS}$ and $K_{ABS}\sm\{n\}$ are not homeomorphic (\cite{Ber71}).
        \item $K_{ABS}$ does not contain non-trivial convergent sequences.
        \item The sequence $(\mu_n)_{n\io}$ of measures on $K_{ABS}$, defined for each $n\io$ by the formula
        \[\quad\quad\mu_n=\frac{1}{2}\big(\delta_{2n}-\delta_{2n+1}\big),\]
        is a disjointly supported JN-sequence consisting of measures having supports of size $2$ (\cite[Proposition 6.4]{MSZ24}). In particular, $C(K_{ABS})$ does not have the Grothendieck property.
        \item There is no space of the form $N_F=\omega\cup\{F\}$, where $F$ is a free filter on $\omega$, carrying a JN-sequence and such that it can be homeomorphically embedded into $K_{ABS}$ (\cite[Example 7.8]{MS24}).
        \item The algebra $\aA_{ABS}$ can be written as the union of a countable increasing chain of proper subalgebras (\cite[Example 4.10]{Sch82}). In particular, it does not have the so-called Nikodym property.
        \item $\aA_{ABS}$ has the so-called Weak Subsequential Separation Property (\cite[Proposition 2.5]{KS13}).
        \item The space $C_p(K_{ABS})$ is isomorphic to the space $C_p(K_{ABS})\oplus\R$ (\cite[Remark 7.9]{MS24}).
    \end{enumerate}
\end{example}

\subsection{Characterizations of symmetric compactifications}

The following theorem characterizes symmetric compactifications of $\omega$ and relates them to JN-sequences and (sq)-sequences. It will be a crucial ingredient of the proof of our main result of this section, Theorem \ref{theorem:sq_jn_2_k}. 

\begin{theorem} \label{theorem:SymmetricCase}
    Let a pair $(A,B)$ be a partition of $\omega$ into two infinite sets. Let $A=\{a_n\colon n\io\}$ and $B=\{b_n\colon n\io\}$ be enumerations. Suppose that $b\omega$ is a compactification of $\omega$. For any $n\io$ define the finitely supported measure $\mu_n$ on $b\omega$ as follows:
    \[\mu_n = \frac{1}{2}\big(\delta_{a_n} - \delta_{b_n}\big).\]
    Then, the following statements are equivalent:
    \begin{enumerate}[(i)]
        \item\label{theorem:SymmetricCase:jnseq} $(\mu_n)_{n\io}$ is a JN-sequence.
        \item\label{theorem:SymmetricCase:sqseq} $(\mu_n)_{n\io}$ is an (sq)-sequence.
        \item\label{theorem:SymmetricCase:nhbd} For every $x \in b\omega \setminus \omega$ and open neighbourhood $U$ of $x$ there is an open neighbourhood $V \subseteq U$ of $x$ such that, for every $n \in \omega$, we have: \[\quad\quad a_n \in V\Longleftrightarrow b_n \in V.\]
        \item\label{theorem:SymmetricCase:net} For every net $(n_{\lambda})_{\lambda\in\Lambda}$ of elements of $\omega$ and every $x \in b \omega \setminus \omega$ we have:
        \[\quad\quad a_{n_\lambda} \xrightarrow[\lambda]{\quad} x \iff b_{n_\lambda} \xrightarrow[\lambda]{\quad} x.\]
        \item\label{theorem:SymmetricCase:id} $b\omega$ is a symmetric compactification of $\omega$ with respect to the partition $(A,B)$ and the homeomorphism $h\colon \overline{A} \rightarrow \overline{B}$ satisfying $h(a_n) = b_n$ for every $n \in \omega$ and $h \restriction (\overline{A} \setminus A) = \id_{\overline{A} \setminus A}$.
        \item\label{theorem:SymmetricCase:kbs} There is a continuous surjection $\varphi\colon K_{ABS}\to b\omega$ such that $\varphi(2n)=a_n$ and $\varphi(2n+1)=b_n$ for every $n\io$.
    \end{enumerate}
\end{theorem}

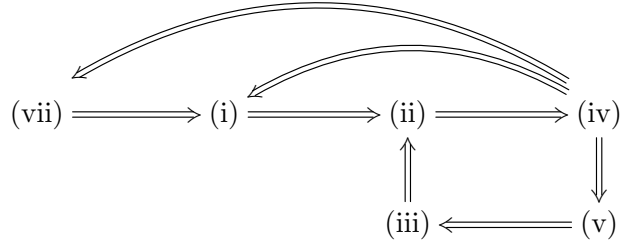
\begin{figure}$$\xymatrixcolsep{4pc}\xymatrix{
\\
\mathrm{(vii)}\ar@{=>}[r]&    \mathrm{(i)}\ar@{=>}[r]&    \mathrm{(ii)}\ar@{=>}[r]&   \mathrm{(iv)}\ar@{=>}[d]\ar@{=>}@/_3pc/@<-5pt>[lll]\ar@{=>}@/_2pc/@<-1pt>[ll]\\
&   &   \mathrm{(iii)}\ar@{=>}[u]&    \mathrm{(v)}\ar@{=>}[l]\\
}$$
\vspace{0mm}
\begin{center}\captionof{figure}[aaaaaa]{The implication scheme for the proof of Theorem \ref{theorem:SymmetricCase} \label{diagram}}
\end{center}\end{figure}

\begin{proof}
    For the scheme of the proof, see Figure \ref{diagram}.

    \medskip
    \ref{theorem:SymmetricCase:jnseq}$\Longrightarrow$\ref{theorem:SymmetricCase:sqseq}: The sequence $(\mu_n)$ is a simple JN-sequence, so by Proposition \ref{prop:simple_jn_sq} it is an (sq)-sequence.

    \medskip

    \ref{theorem:SymmetricCase:nhbd}$\Longrightarrow$\ref{theorem:SymmetricCase:sqseq}: From \ref{theorem:SymmetricCase:nhbd} it follows that for every $x\in b\omega\sm\omega$ we have $x\in\ol{A}\sm\omega$ if and only if $x\in\ol{B}\sm\omega$, and so that $\ol{A}\sm\omega=\ol{B}\sm\omega$. As $b\omega=\ol{A\cup B}=\ol{A}\cup\ol{B}$, we get
    \[b\omega\sm\omega=\ol{A}\sm\omega=\ol{B}\sm\omega.\]
    Obviously, we trivially have
    \[b\omega = A \sqcup B \sqcup (b \omega \setminus \omega).\]
    
    We will show that $(\mu_n)_{n\io}$ is an (sq)-sequence. Fix $g \in C_p(b \omega)$. Let $\eps > 0$ and let $x_1,\dots,x_k \in b \omega$ be distinct. We will find a function $f\in\bigcup_{m\io}\bigcap_{n\ge m}\ker(\mu_n)$ such that $|f(x_i)-g(x_i)|<\eps$ for every $i=1,\ldots,k$. Since the points of $\omega$ are isolated, without loss of generality we can assume that $x_1,\dots,x_k \in b \omega \setminus \omega=\ol{A}\sm\omega$.
    
    We define the function $f$ on $b \omega$ for every $x\in b\omega$ in the following way:
    \[f(x)=\begin{cases}
        g(x),&\text{ if }x\in\ol{A},\\
        g(a_n),&\text{ if }x=b_n\text{ for some }n\io.
    \end{cases}\]
    Then, for every $n\io$ we clearly have $f(a_n)=f(b_n)$, hence
    \[\mu_n(f) = (f(a_n) - f(b_n))/2 = 0\]
    and so $f\in\ker(\mu_n)$. Moreover, for $i=1,\dots,k$ we have $f(x_i) = g(x_i)$ and thus trivially $|f(x_i)-g(x_i)|=0<\eps$, so we are done if we show that $f$ is continuous. 
    
    As any function is continuous at isolated points, we just need to show that $f$ is continuous at any $y \in b \omega \setminus \omega$. Fix such $y$ and let $U'$ be an open neighbourhood of $f(y)$ in $\R$. As $f\rstr \overline{A} = g\rstr\ol{A}$ and $g$ is continuous, there is an open neighbourhood $W$ of $x$ in $\overline{A}$ such that $f[W]=g[W] \subseteq U'$. Let $U$ be an open set in $b \omega$ such that $W = U \cap \overline{A}$. By \ref{theorem:SymmetricCase:nhbd} there is an open subset $V$ of $U$ such that $y \in V$ and, for every $n \in \omega$, $a_n \in V$ if and only if $b_n \in V$. As $f(b_n) = f(a_n)$ for every $n \in \omega$, we get that $f[V \cap B] = f[V \cap A]$. Hence,
    \[f[V] = f\big[V\cap(\ol{A}\cup B)\big]=f[V \cap \ol{A}] \cup f[V \cap B] = f[V \cap \overline{A}] \subseteq f[U \cap \overline{A}] = f[W] \subseteq U',\]
    and thus $f$ is continuous at $y$. It follows that $(\mu_n)_{n\io}$ is an (sq)-sequence, as required.

    \medskip

    \ref{theorem:SymmetricCase:sqseq}$\Longrightarrow$\ref{theorem:SymmetricCase:net}: Suppose that \ref{theorem:SymmetricCase:net} is false. Then, without loss of generality, there are a net $(n_\lambda)_{\lambda\in\Lambda}$ in $\omega$ and a point $x \in b\omega\sm\omega$ such that
    \[a_{n_\lambda} \xrightarrow[\lambda]{\quad} x\quad\text{but}\quad b_{n_\lambda}\centernot{\xrightarrow[\lambda]{\quad}}x.\]
    Note that this implies that $(n_\lambda)_{\lambda \in \Lambda}$ does not admit a constant subnet. By the compactness of $b \omega$ we can pass to a subnet (which we do not relabel) such that there is $y\in b\omega\sm\omega$ with $y\neq x$ and $b_{n_\lambda} \xrightarrow[\lambda]{\quad} y$. Let $g \in C_p(b \omega)$ be any function such that $g(x) = 0$ and $g(y) = 1$. We will show that $g \notin \ol{\bigcup_{m \in \omega} \bigcap_{n \geq m} \ker (\mu_n)}^{\tau_p}$, contradicting \ref{theorem:SymmetricCase:sqseq}. Let
    \[f\in\bigcup_{m\io}\bigcap_{n\ge m}\ker(\mu_n)\]
    be a function such that $-1/3 < f(x) < 1/3$. It follows that $f(a_n) = f(b_n)$ for all but finitely many $n \in \omega$. As $a_{n_\lambda} \rightarrow x$ and $f$ is continuous, there is $\lambda_0\in\Lambda$ such that for all $\lambda \geq \lambda_0$ we have
    \[f(b_{n_\lambda}) = f(a_{n_\lambda}) < 1/3.\]
    But
    \[b_{n_\lambda}\xrightarrow[\lambda \geq \lambda_0]{\quad}y,\]
    so, again by the continuity of $f$, we get $-1/3\le f(y) \leq 1/3$. It follows that
    \[|g(y) - f(y)|=|1-f(y)| \geq 2/3 > 1/3.\]
    Consequently, $f$ does not belong to the open neighbourhood $V\big(g;x,y;1/3\big)$ of $g$. As $f$ was chosen arbitrarily, we see that 
    \[V(g;x,y;1/3)\cap\bigcup_{m\io}\bigcap_{n\ge m}\ker(\mu_n)=\emptyset,\]
    and so $(\mu_n)_{n\io}$ is not an (sq)-sequence, a contradiction.

    \medskip

    \ref{theorem:SymmetricCase:net}$\Longrightarrow$\ref{theorem:SymmetricCase:jnseq}: Assume that $(\mu_n)_{n\io}$ is not a JN-sequence, that is, there are $f\in C(b\omega)$, $\eps>0$, and a strictly increasing sequence $(n_k)_{k\io}$ such that 
    \[\tag{$*$}\mu_{n_k}(f)=\big(f\big(a_{n_k}\big)-f\big(b_{n_k}\big)\big)/2>\eps\]
    for every $k\io$. Pick
    \[x\in\ol{\big\{a_{n_k}\colon\ k\io\big\}}^{b\omega}\sm\omega.\]
    There is a net $(k_\lambda)_{\lambda\in\Lambda}$ in $\omega$ such that $a_{n_{k_\lambda}}\xrightarrow[\lambda]{\quad}x$. It follows by \ref{theorem:SymmetricCase:net} that also $b_{n_{k_\lambda}}\xrightarrow[\lambda]{\quad}x$. Consequently, by the continuity of $f$ we have
    \[\lim_{\lambda\in\Lambda}\Big(f\big(a_{n_{k_\lambda}}\big)-f\big(b_{n_{k_\lambda}}\big)\Big)=f(x)-f(x)=0,\]
    whereas by ($*$) we have
    \[\liminf_{\lambda\in\Lambda}\Big(f\big(a_{n_{k_\lambda}}\big)-f\big(b_{n_{k_\lambda}}\big)\Big)\ge\eps>0,\]
    a contradiction.
    
    \medskip
    %

    \ref{theorem:SymmetricCase:net}$\Longrightarrow$\ref{theorem:SymmetricCase:id}: Let $h\colon \overline{A} \rightarrow b \omega$ be defined by setting $h(a_n) = b_n$ for $n \in \omega$ and $h(x) = x$ for $x \in \overline{A} \setminus A$. By the compactness of $\ol{A}$, we only need to show that $h$ is a continuous injection onto $\ol{B}$. 
    
    For the injectivity of $h$, note that the mapping $h \restriction A$ is trivially injective and onto $B$, and that the mapping $h \restriction(\overline{A} \setminus A)$, being the identity map, is also injective and onto $\overline{A} \setminus A$. Since $(\ol{A}\sm A)\cap B=\emptyset$ 
    (as $B$ is disjoint from $A$ and consists only of isolated points), 
    we have $h[\ol{A}\sm A]\cap h[A]=\emptyset$. It follows that $h$ is injective.
    
    To show the continuity of $h$, let $(x_\lambda)_{\lambda\in\Lambda}$ be a net in $\overline{A}$ which converges to some $x \in \overline{A}$. As every map is continuous at isolated points, we can assume that $x \in \overline{A} \setminus A$, so $h(x)=x$. By passing to a subnet, we can assume that either $x_\lambda \in A$ for every $\lambda\in\Lambda$, or $x_\lambda \in \overline{A} \setminus A$ for every $\lambda\in\Lambda$. In the first case, we find $n_\lambda\in\omega$ for every $\lambda\in\Lambda$ such that $x_\lambda = a_{n_\lambda}$, and use \ref{theorem:SymmetricCase:net} to see that $h(x_\lambda) = b_{n_{\lambda}} \rightarrow x = h(x)$. The second case is also clear as $h \restriction \overline{A} \setminus A$ is the identity. Consequently, $h$ is continuous.
    
    By a similar argument using \ref{theorem:SymmetricCase:net}, we show that $h$ is onto $\overline{B}$.

    \medskip

    \ref{theorem:SymmetricCase:id}$\Longrightarrow$\ref{theorem:SymmetricCase:nhbd}: Let $h$ be as in \ref{theorem:SymmetricCase:id}. Note that
    \[\ol{A}\sm A=h[\ol{A}\sm A]=\ol{B}\sm B,\]
    and so that
    \[b\omega\sm\omega=\ol{A}\sm A=\ol{B}\sm B.\]
    Let $x \in b \omega \setminus \omega$ and let $U$ be an open neighbourhood of $x$ in $b\omega$. Then, the set
    \[W = U \cap \overline{B} \cap h[U \cap \overline{A}]\]
    is an open subset of $\overline{B}$. Set $V = W \cup h^{-1}[W]$. Then:
    \begin{enumerate}[(a)]
        \item $V\sub U$,
        \item $x \in V$,
        \item $V$ is open in $b \omega$,
        \item $V$ is as required in \ref{theorem:SymmetricCase:nhbd}. 
    \end{enumerate}
    Item (a) follows immediately from the definition of $W$ and the fact that $h$ is a bijection. Item (b) is also clear as $x \in W$. 
    
    We show (c). Note that, by the definition of $h$, we get $h^{-1} [W \setminus B] = W \setminus B$ and $h^{-1} [W \cap B] \cap \overline{B} = \emptyset$, hence
    \[V \cap \overline{B} = \big(W \cup h^{-1} [W \cap B] \cup h^{-1} [W \setminus B]\big) \cap \overline{B} =\]
    \[=\big(W \cap \overline{B}\big)\cup\big((W\sm B)\cap\ol{B}\big)=W\cap\ol{B}.\]
    Similarly, as $W\cap\overline{A}=U\cap(\overline{A}\sm A)$ and 
    \[h^{-1}[W]\cap\overline{A}=\big(U\cap(\overline{A}\sm A)\big)\cup\big(h^{-1}[U\cap B]\cap U\cap\ol{A}\big),\]
    we get that $W\cap\overline{A}\sub h^{-1}[W]\cap\ol{A}$, and so
    \[V\cap\overline{A}=(W\cup h^{-1}[W])\cap\overline{A}=\]
    \[=(W\cap\overline{A})\cup(h^{-1}[W]\cap\overline{A})=h^{-1}[W] \cap \overline{A}.\]
    Hence, $V \cap \overline{A}$ is open in $\ol{A}$ and $V \cap \overline{B}$ is open in $\overline{B}$. Let $y\in V$. Suppose, for a contradiction, that $y \in \overline{b \omega \setminus V}$. Then, as
    \[b\omega\sm V=(b\omega\sm V)\cap b\omega=(b\omega\sm V)\cap(\ol{A}\cup\ol{B})=\big((b\omega\sm V)\cap\ol{A}\big)\cup\big((b\omega\sm V)\cap\ol{B}\big),\]
    by the relative openness of $V$ in both $\ol{A}$ and $\ol{B}$, we have
    \[y \in \overline{(b \omega \setminus V) \cap \overline{A}}=(b \omega \setminus V) \cap \overline{A}\sub b\omega\sm V,\]
    or we have
    \[y \in \overline{(b \omega \setminus V) \cap \overline{B}}=(b \omega \setminus V) \cap \overline{B}\sub b\omega\sm V,\]
    hence $u\in b\omega\sm V$, which contradicts the fact that $y\in V$. Thus, (c) is proved.

    Finally, we prove (d). Let $n\io$ be such that $a_n\in V$. Then, $a_n\in W$ or $a_n\in h^{-1}[W]$. But if $a_n\in W$, then $a_n\in\ol{B}$, which is impossible as $a_n\in A$ and $A\cap\ol{B}=\emptyset$. Hence, $a_n\in h^{-1}[W]$, which implies that $b_n=h(a_n)\in W\sub V$, as required. On the other hand, let $n\io$ be such that $b_n\in V$. Then, $b_n\in W$ or $b_n\in h^{-1}[W]$. But if $b_n\in h^{-1}[W]$, then $b_n\in\ol{A}$, which is again impossible. So, $b_n\in W$ and hence $b_n\in U\cap\ol{B}$ and $b_n\in h[U\cap\ol{A}]$. The former implies that $a_n\in h^{-1}[U\cap\ol{B}]$, while the latter yields that $a_n\in U\cap\ol{A}$. Consequently, $a_n\in h^{-1}[W]\sub V$, again as required.

    \medskip

    \ref{theorem:SymmetricCase:kbs}$\Longrightarrow$\ref{theorem:SymmetricCase:jnseq}. Fix a continuous surjection $\varphi\colon K_{ABS}\to b\omega$ as in \ref{theorem:SymmetricCase:kbs}. Let $f\in C(b\omega)$. Note that $f\circ\varphi\in C(K_{ABS})$. Set $\mu_n^{ABS} = \frac{1}{2} \left(\delta_{2n} - \delta_{2n+1} \right) \in M_f(K_{ABS})$ for $n \in \omega$; then $(\mu_n^{ABS})_{n \in \omega}$ is weak* null by Example \ref{ex:ABS}.(2). We have
    \[\lim_{n\to\infty}\mu_n(f)=\lim_{n\to\infty}\big(f(a_n)-f(b_n)\big)/2=\]
    \[=\lim_{n\to\infty}\big(f(\varphi(2n))-f(\varphi(2n+1))\big)/2=\lim_{n\to\infty}\mu_n^{ABS}(f\circ\varphi)=0.\]
    Consequently, $(\mu_n)_{n\io}$ is a JN-sequence.

    \medskip

    \ref{theorem:SymmetricCase:net}$\Longrightarrow$\ref{theorem:SymmetricCase:kbs}. We construct a mapping $\varphi\colon K_{ABS}\to b\omega$ as follows. First, for each $n\io$ set $\varphi_0(2n)=a_n$ and $\varphi_1(2n+1)=b_n$, as required. Then, extend the (continuous) surjections
    \[\varphi_0\colon\EE\to A\quad\text{and}\quad\varphi_1\colon\OO\to B\]
    to the continuous surjections
    \[\Phi_0\colon\beta\EE\to\ol{A}\quad\text{and}\quad\Phi_1\colon\beta\OO\to\ol{B},\]
    respectively. 
    
    For each $x\in K_{ABS}\sm\omega$ we have
    \[\tag{$**$}\Phi_0(x)=\Phi_1(x).\]
    To see this, fix $x\in K_{ABS}\sm\omega$ and let $(n_\lambda)_{\lambda\in\Lambda}$ be a net in $\omega$ such that $\lim_{\lambda\in\Lambda}2n_\lambda=x$ in $K_{ABS}$. Note that also $\lim_{\lambda\in\Lambda}(2n_\lambda+1)=x$ in $K_{ABS}$. Then, by the continuity of $\Phi_0$ and $\Phi_1$, we have
    \[\lim_{\lambda\in\Lambda}a_{{n_\lambda}}=\lim_{\lambda\in\Lambda}\Phi_0(2n_\lambda)=\Phi_0(x)\]
    and
    \[\lim_{\lambda\in\Lambda}b_{{n_\lambda}}=\lim_{\lambda\in\Lambda}\Phi_1(2n_\lambda+1)=\Phi_1(x).\]
    Since by \ref{theorem:SymmetricCase:net} we have $\lim_{\lambda\in\Lambda}a_{{n_\lambda}}=\lim_{\lambda\in\Lambda}b_{{n_\lambda}}$, we get that $\Phi_0(x)=\Phi_1(x)$.

    For each $x\in K_{ABS}$ set
    \[\varphi(x)=\begin{cases}
        \Phi_0(x),&\text{ if }x\in\EE,\\
        \Phi_1(x),&\text{ if }x\in K_{ABS}\sm\EE.
    \end{cases}\]
    Then, $\varphi\colon K_{ABS}\to b\omega$ is a well-defined surjection. We claim that $\varphi$ is continuous---it is enough to check this at points of $K_{ABS}\sm\omega$. So, fix $x\in K_{ABS}\sm\omega$ and let $(x_\lambda)_{\lambda\in\Lambda}$ be any net in $K_{ABS}$ convergent to $x$. Let $\Lambda_0=\{\lambda\in\Lambda\colon x_\lambda\in\EE\}$ and $\Lambda_1=\Lambda\sm\Lambda_0$. For $i\in\{0,1\}$, if $\Lambda_i$ is cofinal in $\Lambda$ and so $\lim_{\lambda\in\Lambda_i}x_\lambda=x$, then by the continuity of $\Phi_i$ and ($*$) we have
    \[\lim_{\lambda\in\Lambda_i}\varphi(x_\lambda)=\lim_{\lambda\in\Lambda_i}\Phi_i(x_\lambda)=\Phi_i(x)=\varphi(x).\]
    It follows that
    \[\lim_{\lambda\in\Lambda}\varphi(x_\lambda)=\varphi(x),\]
    and, consequently, $\varphi$ is continuous at $x$.
\end{proof}


The following corollary justifies the statements that the space $K_{ABS}$ is a maximal symmetric compactification of $\omega$ whereas the interval $[0,\omega]$ is a minimal one.

\begin{corollary}\label{cor:symm_max_min}
    Let $b\omega$ be a symmetric compactification of $\omega$ with respect to some partition $(A,B)=\big(\{a_n\}_{n\io},\{b_n\}_{n\io}\big)$ of $\omega$ and some homeomorphism $h\colon\ol{A}\to\ol{B}$ such that $h(a_n)=b_n$ for all $n\io$. Then, there are continuous surjections
    \[\varphi\colon K_{ABS}\to b\omega\quad\text{and}\quad\psi\colon b\omega\to[0,\omega]\]
    such that
    \[\varphi(2n)=a_n=\psi(a_n)\quad\text{and}\quad\varphi(2n+1)=b_n=\psi(b_n)\]
    for all $n\io$.
\end{corollary}
\begin{proof}
    The existence of the mapping $\varphi$ follows from Theorem \ref{theorem:SymmetricCase}. The mapping $\psi\colon b\omega\to[0,\omega]$ is defined by $\psi(n)=n$ for each $n\io$, and $\psi(x)=\omega$ for all $x\in b\omega\sm\omega$; it is immediate that $\psi$ is continuous.
\end{proof}

Recall that a topological space $X$ is \textit{zero-dimensional} if it the Boolean algebra $Clopen(X)$ is a basis of the topology of $X$. The following corollary translates equivalence \ref{theorem:SymmetricCase:id}$\Leftrightarrow$\ref{theorem:SymmetricCase:kbs} of Theorem \ref{theorem:SymmetricCase} into the Boolean-theoretic language; its proof is a routine application of the Stone duality (see \cite{Kop89}). The further two corollaries are its immediate consequences.

\begin{corollary}
    Let $b\omega$ be a zero-dimensional compactification of $\omega$ and set $\aA=Clopen(b\omega)$. Let $(A,B)=\big(\{a_n\}_{n\io},\{b_n\}_{n\io}\big)$ be a partition of $\omega$. Then, the following are equivalent:
    \begin{enumerate}[(i)]
        \item $b\omega$ is symmetric with respect to $(A,B)$,
        \item there is a Boolean monomorphism $\Phi\colon\aA\to\aA_{ABS}$ such that $\Phi(\{a_n\})=\{2n\}$ and $\Phi(\{b_n\})=\{2n+1\}$ for every $n\io$.
    \end{enumerate}
\end{corollary}

\begin{corollary}\label{cor:algebra_chain}
    If $b\omega$ is a zero-dimensional symmetric compactification of $\omega$, then the Boolean algebra $Clopen(b\omega)$ is the union of a countable increasing chain of proper subalgebras.
\end{corollary}

\begin{remark}
    It follows from Corollary \ref{cor:algebra_chain} and \cite[Proposition 4.6]{Sch82} that for a zero-dimensional symmetric compactification $b\omega$ the Boolean algebra $Clopen(b\omega)$ does not have the so-called Nikodym property.
\end{remark}

\begin{corollary}
    If $\aA$ is a Boolean subalgebra of $\aA_{ABS}$ such that $\{n\}\in\aA$ for every $n\io$, then the Stone space $St(\aA)$ of $\aA$ is a symmetric compactification of $\omega$.
\end{corollary}

%
%

\subsection{JN-sequences and (sq)-sequences with bounded sizes of supports}

We are now going to use Theorem \ref{theorem:SymmetricCase} to characterize the existence of JN-sequences and (sq)-sequences with uniformly bounded sizes of supports. First, we make a simple observation that the only weak* cluster point of an (sq)-sequence in $M_f(K)$ can be the zero measure.

\begin{lemma}\label{lemma:cluster_point}
    Let $K$ be a compact space. Let $(\mu_n)_{n\io}$ be a norm-bounded sequence in $M(K)$ and set $\xX = \bigcup_{m \in \omega} \bigcap_{n \geq m}\ker(\mu_n)$. Let $\mu\in M(K)$ be a weak* cluster point of the set $\{\mu_n\colon n\io\}$. Then:
    \begin{enumerate}[(a)]
        \item $\xX\sub\ker(\mu)$,
        \item if $(\mu_n)_{n\io}$ is an (sq)-sequence and $\mu \in M_f(K)$, then $\mu = 0$.
    \end{enumerate}
\end{lemma}
\begin{proof}
    For every $f \in\xX$ we have $\mu_n(f) = 0$ for all but finitely many $n \io$, hence $\mu(f) = 0$ and so (a) follows. To see (b), note that if $\mu \in M_f(K)$, then it is $\tau_p$-continuous, and hence, by (a) and the definition of (sq)-sequences, we have
    \[\mu=\mu\rstr C_p(K)=\mu \rstr \overline{\xX}^{\tau_p}=0.\]
\end{proof}

Recall that for any compact space $K$ and $k \in \omega$, the set $M^1_k(K)$ is a weak* compact subset of $M_f(K)$ (as it is the image of the compact space $K^k \times B_{\ell_1(k)}$ under the continuous map $(x_1,\dots,x_k,a_1,\dots,a_k) \mapsto \sum_{i=1}^k a_i \delta_{x_i}$). We have already seen in Proposition \ref{prop:simple_jn_sq} that every simple JN-sequence is an (sq)-sequence. In what follows we show that having uniformly bounded sizes of supports is a sufficient condition for an (sq)-sequence to be a JN-sequence.

\begin{proposition}\label{prop:sqk_jnk}
    Let $K$ be a compact space. Let $(\mu_n)_{n\io}$ be a normalized (sq)-sequence in $M_k(K)$ for some $k \in \N$. Then $(\mu_n)_{n\io}$ is a JN-sequence.
\end{proposition}

\begin{proof}
    As $M^1_k(K)$ is a compact subset of $M_f(K)$, we get that all the weak* cluster points of $\{\mu_n\colon n\io\}$ are in $M_f(K)$. Hence, by Lemma \ref{lemma:cluster_point}.(b), $\{\mu_n\colon n\io\}$ has the unique weak* cluster point $0$. As every sequence in a compact space with a unique cluster point is necessarily a convergent sequence, the claim follows.
\end{proof}

For the proof of the next theorem, we will need the following important result from \cite{MSZ24} that allows us to reduce a JN-sequence with uniformly bounded sizes of supports to a JN-sequence with size of support being 2.

\begin{theorem}[{\cite[Theorem 1.4]{MSZ24}}]\label{theorem:JNseq_size2}
    Let $X$ be a Tychonoff space. If $X$ carries a JN-sequence from $M_k(X)$ for some $k\io$, then there is a disjointly supported JN-sequence on $X$ from $M_2(X)$.
\end{theorem}

Putting everything together, we get the main result of the paper.


\begin{theorem}\label{theorem:sq_jn_2_k}
    Let $K$ be a compact space. The following are equivalent:
    \begin{enumerate}
        \item\label{theorem:sq_jn_2_k:sq2} There is an (sq)-sequence in $M_2(K)$.
        \item\label{theorem:sq_jn_2_k:sqk} There is an (sq)-sequence in $M_k(K)$ for some $k\io$.
        \item\label{theorem:sq_jn_2_k:jn2} There is a JN-sequence in $M_2(K)$.
        \item\label{theorem:sq_jn_2_k:jnk} There is a JN-sequence in $M_k(K)$ for some $k\io$.
        \item\label{theorem:sq_jn_2_k:psc} There exists a copy of a symmetric compactification of $\omega$ in $K$.
    \end{enumerate}
\end{theorem}
\begin{proof}
    Implications (\ref{theorem:sq_jn_2_k:sq2})$\Rightarrow$(\ref{theorem:sq_jn_2_k:sqk}) and (\ref{theorem:sq_jn_2_k:jn2})$\Rightarrow$(\ref{theorem:sq_jn_2_k:jnk}) are obvious. Implication (\ref{theorem:sq_jn_2_k:sqk})$\Rightarrow$(\ref{theorem:sq_jn_2_k:jnk}) follows from Proposition \ref{prop:sqk_jnk}, implication (\ref{theorem:sq_jn_2_k:jnk})$\Rightarrow$(\ref{theorem:sq_jn_2_k:jn2}) from Theorem \ref{theorem:JNseq_size2}, and implication (\ref{theorem:sq_jn_2_k:psc})$\Rightarrow$(\ref{theorem:sq_jn_2_k:sq2}) from Theorem \ref{theorem:SymmetricCase} (implication \ref{theorem:SymmetricCase:id}$\Rightarrow$\ref{theorem:SymmetricCase:sqseq}).

    For implication (\ref{theorem:sq_jn_2_k:jn2})$\Rightarrow$(\ref{theorem:sq_jn_2_k:psc}), note that by \cite[Lemma 6.6]{MSZ24} there is a disjointly supported JN-sequence $(\mu_n)_{n\io}$ in $M_2(K)$. By going to a subsequence, we may assume that $(\mu_n)_{n\io}$ is discretely supported and even simple, hence the implication follows again by Theorem \ref{theorem:SymmetricCase} (implication \ref{theorem:SymmetricCase:jnseq}$\Rightarrow$\ref{theorem:SymmetricCase:id}). 
\end{proof}

\section{Totally asymmetric compactifications of $\omega$} \label{sec:Asym}

We now focus on totally asymmetric compactifications of $\omega$. We will show that such compactifications are closely related to the Separable Quotient Problem for $C_p(X)$-spaces in the sense that if there exists a compact space $K$ such that $C_p(K)$ does not admit a separable quotient, then there exist such $K$ which is additionally a totally asymmetric compactification of $\omega$.


As was mentioned before, the presence of a convergent sequence in a Tychonoff space $X$ is sufficient for $C_p(X)$ to have a separable quotient. K\k{a}kol and \'{S}liwa \cite{KS18} showed that the presence of a copy of $\beta \omega$ in $X$ also suffices (see also \cite{BKS19}). This was achieved by showing that $\beta \omega$ has the following property, later named 2DCP in \cite{KKS26}, which implies that the space $C_p(X)$ has a separable quotient. 

\begin{definition}\label{def:2dcp}
    A Tychonoff space $X$ has \textit{the Two-Disjoint-Copies Property}, in short \textit{2DCP}, if there exists a sequence  $(K_n)_{n \in \omega}$ of non-empty compact subsets of $X$ such that for every $n \in \omega$ the set $K_n$ contains two disjoint subsets homeomorphic to $K_{n+1}$.
\end{definition}

It is immediate that both the Cantor space $2^\omega$ and $\beta \omega$ have the 2DCP, but $[0,\omega]$ does not. For a systematic study of 2DCP, see \cite{KKS26}. 

In the following result we give a condition on a compactification of $\omega$ which is somewhat orthogonal to being totally asymmetric and which is sufficient for the compact space to have 2DCP.

\begin{proposition} \label{prop:KakSli2018}
    Let $K$ be a compact space satisfying the following property: There exists a discrete set $M_0\in\ctblsub{K}$ such that for every $M \in [M_0]^\omega$ there are $A,B \in [M]^\omega$ such that $\overline{A} \cap \overline{B} = \emptyset$ and $\overline{A}$ and $\overline{B}$ are homeomorphic. Then, $K$ has 2DCP and consequently $C_p(K)$ has a separable quotient.
\end{proposition}

\begin{proof}
    Use the assumption inductively to construct a sequence $(K_n)_{n\io}$ as in the definition of 2DCP. That 2DCP of $K$ implies that $C_p(K)$ has a separable quotient follows from \cite[Theorem 4]{KS18}.
\end{proof}

Now, we are all prepared to reduce the Separable Quotient Problem for spaces $C_p(K)$, $K$ compact, to the case when $K$ is a totally asymmetric compactification of $\omega$, that is, to prove Theorem \ref{theorem:mainD} from Introduction.

\begin{remark}
    Note that the definition of totally asymmetric compactifications of $\omega$ can be easily transferred any countable discrete sets $M$.
\end{remark}

\begin{theorem} \label{theorem:totasym}
    Suppose that $K$ is a compact space such that $C_p(K)$ does not have a separable quotient. Then, there is a discrete set $M\in\ctblsub{K}$ such that $\overline{M}$ is a totally asymmetric compactification of $M$ and $C_p(\overline{M})$ does not have a separable quotient.
\end{theorem}

\begin{proof}    
    Let $M_0$ be any countable infinite discrete subset of $K$. Note that $K$ cannot contain a symmetric compactification of integers, as this would imply that it carries an (sq)-sequence (by Theorem \ref{theorem:sq_jn_2_k}), and so that
    $C_p(K)$ has a separable quotient. Hence, using the characterization from Theorem \ref{theorem:SymmetricCase}.\ref{theorem:SymmetricCase:id}, we have:
    \begin{itemize}
        \item[($*$)] for every disjoint $A,B \in [M_0]^\omega$ and every homeomorphism $h\colon\overline{A} \rightarrow \overline{B}$ there is $x \in \overline{A} \setminus A$ such that $h(x) \neq x$. 
    \end{itemize}
    We claim that there is $M \in [M_0]^\omega$ such that $\overline{M}$ is a totally asymmetric compactification of $M$. Suppose, for a contradiction, that this is not the case. Then:
    \begin{itemize}
        \item[($**$)] for every $M \in [M_0]^\omega$ there are disjoint subsets $A,B \in [M]^\omega$ such that $\overline{A}$ and $\overline{B}$ are homeomorphic. 
    \end{itemize}
    
    We will show that this implies that the condition from Proposition \ref{prop:KakSli2018} is satisfied.
    Let $M \in [M_0]^\omega$. We use $(**)$ to find disjoint $A_0, B_0 \in [M]^\omega$ such that $\overline{A_0}$ and $\overline{B_0}$ are homeomorphic. Let $h\colon\overline{A_0} \rightarrow \overline{B_0}$ be a homeomorphism. By $(*)$ there is $x \in \overline{A_0} \setminus A_0$ such that $h(x) \neq x$. We find open neighbourhoods $U_0$ of $x$ in $\overline{A_0}$ and $V_0$ of $h(x)$  in $\overline{B_0}$ such that $\overline{U_0} \cap \overline{V_0} = \emptyset$. The sets $U = U_0 \cap h^{-1}[V_0]$ and $V = V_0 \cap h[U_0]$ are again neighbourhoods of $x$ in $\overline{A_0}$ and $h(x)$ in $\overline{B_0}$, respectively, such that $h[\overline{U}] =\overline{V}$ and $\overline{U} \cap \overline{V} = \emptyset$. Set $A = A_0 \cap U$ and $B = B_0 \cap V$. Then $A,B$ are infinite by the density of $M$ in $\overline{M}$ and $h[A] = B$. It follows that $h \restriction \overline{A}$ is a homeomorphism of $\overline{A}$ onto $\overline{B}$. Further, since $\overline{A} \subseteq \overline{U}$ and $\overline{B} \subseteq \overline{V}$, we have $\overline{A} \cap \overline{B} = \emptyset$.
    
    Consequently, we have proved the following: For every $M \in [M_0]^\omega$ there are infinite disjoint subsets $A,B \in [M]^\omega$ such that $\overline{A}$ and $\overline{B}$ are homeomorphic and $\overline{A} \cap \overline{B} = \emptyset$.    
    But then Proposition \ref{prop:KakSli2018} implies that $C_p(K)$ has a separable quotient, a contradiction. Hence, $(**)$ cannot be true, and so there exists $M \in [M_0]^\omega$ such that $ \overline{M}$ is a totally asymmetric compactification of $M$, as required.
\end{proof}

\begin{remark}
    Note that in the above proof we have found the set $M$ inside an \textit{arbitrary} countable infinite discrete set $M_0$. This means that if a space $C_p(K)$ does not admit a separable quotient, then $K$ is saturated with totally asymmetric compactifications $bM$ such that $C_p(bM)$ does not admit a separable quotient, that is, every infinite closed subset of $K$ contains such a compactification $bM$.
\end{remark}

%

\section*{AI statement}

No artificial intelligence tools have been used at any stage of conducting the presented research or during the preparation and edition of the paper.

\end{document}